\documentclass{amsart}
\usepackage{amssymb,amsmath}
\usepackage{ytableau}
\usepackage[enableskew]{youngtab}
\usepackage{young}
\usepackage{tikz,pgf}

\newtheorem{theorem}{Theorem}
\newtheorem{corollary}{Corollary}
\newtheorem{lemma}{Lemma}
\newtheorem{definition}{Definition}

\newtheorem{conjecture}{Conjecture}
\newtheorem{proposition}{Proposition}

\begin{document}

\title{Restricted $k$-color partitions}
\author{William J. Keith}
\maketitle

\begin{abstract}

We generalize overpartitions to $(k,j)$-colored partitions: $k$-colored partitions in which each part size may have at most $j$ colors.  We find numerous congruences and other symmetries.  We use a wide array of tools to prove our theorems: generating function dissections, modular forms, bijections, and other combinatorial maps.  We find connections to divisor sums, the Han/Nekrasov-Okounkov hook length formula and a possible approach to a finitization, and other topics, suggesting that a rich mine of results is available.

\end{abstract}

\section{Introduction}

Major MacMahon is generally credited with having defined overpartitions: partitions in which the last instance of a given part may be overlined or not.  The overpartitions of 3 are \[ 3, \overline{3}, 2+1, \overline{2}+1, 2+\overline{1}, \overline{2}+ \overline{1}, 1+1+1, 1+1+\overline{1} .\]

The generating function for $\overline{p}(n)$, the number of overpartitions of $n$, is \[ \overline{P}(q) := \sum_{n=0}^\infty \overline{p}(n) q^n = \prod_{k=1}^\infty \frac{1+q^k}{1-q^k} = \prod_{k=1}^\infty \frac{1-q^{2k}}{(1-q^k)^2}. \]

In 2004 \cite{CoLo}, Corteel and Lovejoy revisited overpartitions, showing that they are a simple and useful object that naturally arise from the theory of hypergeometric series, and spurring much recent work.  Numerous authors since have found, as Corteel and Lovejoy expected, rich structure in overpartitions, such as congruences and modular form identities: see (\cite{CHSZ}, \cite{Kim}, \cite{XiaYao}).  Overpartitions can also be restricted, like normal partitions, as to the sizes and frequencies of parts appearing, or which parts can be overlined; singular overpartitions are one such variant recently studied by Andrews, Hirschhorn and Sellers (\cite{GEA2}, \cite{HirschSell2}).

A separate object, $k$-colored partitions (also known as multipartitions), are those in which parts may appear in $k$ different types, with the order of colors not mattering (i.e., by convention we may list colors within a part in weakly descending order).  For instance, denoting colors by subscripts, the 2-colored partitions of 3 are \[ 3_2, 3_1, 2_2+1_2, 2_2 + 1_1, 2_1+1_2, 2_1+1_1, 1_2+1_2+1_2, 1_2+1_2+1_1,1_2+1_1+1_1, 1_1+1_1+1_1 .\]

The generating function of $c_k(n)$, the number of $k$-colored partitions of $n$, is \[ C_k(q) := \sum_{n=0}^\infty c_k(n) q^n = \prod_{n=1}^\infty \frac{1}{(1-q^n)^k}. \]

Much work has been done on $k$-colored partitions as well, including many congruences and their properties as modular forms for various $k$.

In the language of colored partitions, overpartitions would be 2-colored partitions in which only a single color may appear for a given size of part.  The motivation of this paper is to consider the natural generalization of $k$-colored partitions in which at most $j$ colors can appear for a given part size.  We show that these objects, which we call $(k,j)$-\emph{colored partitions}, possess their own rich body of symmetries, and are susceptible to both combinatorial and algebraic analysis.  We find numerous congruences, connections to divisor sums in the work of Dilcher and Andrews (\cite{Dilcher}, \cite{GEA1}), and special cases of interest.  There is certainly more to be mined here.

\section{Definitions}

We introduce the formal definitions.

When for two power series $f(q) = \sum_{n=0}^\infty a(n) q^n$ and $g(q) = \sum_{n=0}^\infty b(n) q^n$ we write $f(q) \equiv_m g(q)$, we mean $a(n) \equiv b(n) \pmod{m}$ for all $n$.

We will use the standard notation \[ (a;q)_n = (1-a)(1-aq)\dots(1-aq^{n-1}) , \quad (a;q)_\infty = \lim_{n\rightarrow \infty} (a;q)_n, \quad (q)_\infty := (q;q)_\infty .\]

A weakly decreasing sequence of positive integers $\lambda = (\lambda_1, \dots, \lambda_j)$ partitions $n$, denoted $\lambda \vdash n$, if $\sum_{i=1}^j \lambda_i = n$.  It will be convenient for us on occasion to write partitions using the frequency notation $\lambda = 1^{e_1}2^{e_2}3^{e_3}\dots$, meaning that there are $e_1$ ones in $\lambda$, $e_2$ twos, et cetera.

The Ferrers diagram of a partition $(\lambda_1, \dots , \lambda_j)$ is a stack of unit-size squares justified to the origin in the first quadrant, having $\lambda_i$ squares in the $i$-th column.  Thus for instance the Ferrers diagram of $(4,4,2,1,1)$ is

$$\young(\hfil\hfil:::,\hfil\hfil:::,\hfil\hfil\hfil::,\hfil\hfil\hfil\hfil\hfil)$$

The \emph{hook length} $h_{ij}$ of the square with upper right corner at $(i,j)$ in the plane is the sum of the number of squares directly above and directly to the right of the square, plus the 1 for the square itself.  The hook lengths are marked in the diagram below.

$$\young(21:::,32:::,541::,87421)$$

The \emph{conjugate} of a partition $\lambda$ is the partition with Ferrers diagram that of $\lambda$ reflected across the diagonal.  The conjugate of $(4,4,2,1,1)$ is $5,3,2,2)$.

$$\young(\hfil:::,\hfil:::,\hfil\hfil::,\hfil\hfil\hfil\hfil,\hfil\hfil\hfil\hfil)$$

A partition fixed under conjugation is \emph{self-conjugate}.

The set of $k$-colored whole numbers is $\mathbb{N}_k = \{ a_b \vert a,b \in \mathbb{N}, 1 \leq b \leq k \}$, with $a_b \geq c_d \Leftrightarrow (a > c) \, \text{OR} \, (a=c \, \text{AND} \, b \geq d)$ and $\vert a_b + c_d \vert = a+c$, extended linearly.

The main object of interest in this paper is

\begin{definition} A $k$-colored partition of an integer $n$ is a weakly decreasing sequence $\lambda = (\lambda_1,\dots,\lambda_\ell)$ with all $\lambda_i \in \mathbb{N}_k$ such that $\vert \lambda_1 + \dots + \lambda_\ell \vert = n$.  A $(k,j)$-colored partition of $n$ is a $k$-colored partition in which at most $j$ colors appear for any given size of part.
\end{definition}

Denote the number of $(k,j)$-colored partitions of $n$ by $c_{k,j}(n)$.  The generating function of $c_{k,j}(n)$ is 

\begin{theorem} 
\begin{multline*}C_{k,j}(q) := \sum_{n=0}^\infty c_{k,j}(n) q^n = \prod_{n=1}^\infty \left( 1 + \frac{\binom{k}{1} q^n}{1-q^n} + \frac{\binom{k}{2} q^{2n}}{(1-q^{n})^2} + \dots + \frac{\binom{k}{j}q^{jn}}{(1-q^n)^j} \right) \\
= \frac{1}{{(q)_\infty}^j} \prod_{n=1}^\infty \left( \sum_{i=0}^j \binom{k}{i} (1-q^n)^{j-i} q^{in} \right) = \frac{1}{{(q)_\infty}^j} \prod_{n=1}^\infty \left( \sum_{i=0}^j \binom{k-j+i-1}{i} q^{i n} \right).
\end{multline*}
\end{theorem}

\begin{proof} Each $\binom{k}{i}$ term in the first line arises from a choice, for part size $n$, of $i$ of the $k$ available colors, adding parts of size $n$ and the chosen colors to the partition, and as many more of the same size as desired.

The second line follows from collecting terms over a common denominator and expanding the sum $\sum_{i=0}^j \binom{k}{i} (1-q^n)^{j-i} q^{i n}.$
\end{proof}

\section{$(k,1)$-colored partitions}\label{K1}

One special case of interest are the partitions which might be described as generalizing overpartitions by allowing the last instance of a given part to be overlined, doubly overlined, etc.  Given $k$ colors, these are the partitions in which one of $k$ colors may be used per size of part.  The algebraic properties of the generating function and the combinatorial properties of the partitions are both usefully analyzable.

The generating function for $(k,1)$-colored partitions is

\begin{equation}\label{KOne}C_{k,1}(q) = \prod_{n=1}^\infty \left( 1+ \frac{k q^n}{1-q^n} \right) = \prod_{n=1}^\infty \frac{1+(k-1) q^n}{1-q^n} =: \sum_{n=0}^\infty f_n (k-1) q^n.
\end{equation}

The polynomials $f_n(u)$ are easily described: their coefficients count a certain subset of 2-colored partitions of $n$, and as a whole, the polynomial is a simple sum concerning the number of outer corners, or part sizes, in the partitions of $n$.  Listed as an irregular triangular array, they form OEIS sequence A008951 \cite{OEIS1}.  We summarize and prove these statements below.

\begin{theorem}\label{KOneFn} Let $f_n(u)$ be defined as above.  Say $f_n(u) = f_{n,0} + f_{n,1} u + \cdots f_{n,t_n} u^{t_n}$.  We have $\binom{t_n+1}{2} \leq n < \binom{t_n+2}{2}$, i.e., $n$ is at least the $t_n$-th positive triangular number but not the $(t_n+1)$-st. Further, we have several combinatorial descriptions of the $f_{n,i}$:

\begin{enumerate}
\item $f_{n,i}$ is the number of overpartitions of $n$ (i.e., $(2,1)$-colored partitions) in which exactly $i$ colors are marked;
\item $f_{n,i} = \sum_{\lambda \vdash n} \binom{\ell_o(\lambda)}{i}$, where $\ell_0(\lambda)$ is the number of sizes of parts appearing in $\lambda$;
\item $f_n = \sum_{\lambda \vdash n} (1+u)^{\ell_o(\lambda)}$, i.e. $f_n(k-1) = \sum_{\lambda \vdash n} k^{\ell_0(\lambda)}$;
\item $f_{n,i}$ is the number of partitions of $n-\binom{i+1}{2}$ with two colors of parts 1 through $i$.
\end{enumerate}
\end{theorem}

\begin{proof}  Note that equivalently, $t_n$ is the maximum number of distinct part sizes in a partition of $n$.

Statement (1) simply interprets the generating function by noting that for each size of part, we may select unmarked or $u$-marked parts in the product which defines the partition.

Statement (2) follows from Statement (1) and the observation that in a partition with $\ell_0(\lambda)$ part sizes in which we are to mark $i$ of them, there are $\binom{\ell_0(\lambda)}{i}$ choices.

Statement (3) is the binomial summation of Statement (2) (and is derived differently as Example 8 in \cite{Fine}, p. 39).

Statement (4): We construct a bijection between the set of overpartitions with $i$ parts marked, and the set of partitions of $n-\binom{i+1}{2}$ with two colors possible for parts 1 through $i$ (of which both may be used for a given part size).

From a partition of $n-\binom{i+1}{2}$ with two colors of parts 1 through $i$, take those parts of the second color and conjugate to form a partition with at most $i$ parts.  To these, add the triangle $i$, $i-1$, et cetera, to form a partition into exactly $i$ distinct parts.  Now take the union of these parts with the uncolored parts remaining in the starting partition, and consider as having the second color all parts of sizes originating in the conjugate-plus-triangle partition, and as having the first color all parts of any other size.

The reverse bijection from a $(2,1)$-colored partition is to take one of each part size of the second color to form a partition into distinct parts, uncolor all remaining parts, subtract a triangle of size equal to the number of colored parts taken, conjugate, and take the union of these with the uncolored parts.

$$\tiny\young(\blacksquare::::::,\hfil\blacksquare:::::,\hfil\hfil:::::,\hfil\hfil\hfil\hfil\hfil::,\hfil\hfil\hfil\hfil\hfil\hfil\hfil) \, \leftrightarrow \, 
\tiny\young(\blacksquare::,\hfil\blacksquare:,\hfil\hfil:,\hfil\hfil:,\hfil\hfil\blacksquare) + 
\tiny\young(\hfil\hfil\hfil:,\hfil\hfil\hfil\hfil)
\, \leftrightarrow \,
\tiny\young(\blacksquare\blacksquare,\hfil\hfil)
 +
 \tiny\young(\hfil\hfil\hfil:,\hfil\hfil\hfil\hfil)
 \, \leftrightarrow \, 
 \tiny\young(\hfil\hfil\hfil\blacksquare\blacksquare:,\hfil\hfil\hfil\hfil\hfil\hfil)$$

\end{proof}

Statement (3) of Theorem \ref{KOneFn} is provocative when we consider the Han/Nekrasov-Okounkov hook length formula, which expands the product

\[ \sum_{n=0}^\infty p_n(b) q^n := \prod_{n=1}^\infty (1-q^n)^{b-1} = \sum_{n=0}^\infty q^n \sum_{\lambda \vdash n} \prod_{h_{ij} \in \lambda} (1-\frac{b}{h_{ij}^2}), \]

\noindent where the $h_{ij}$ are the hooklengths that appear in the Ferrers diagram of $\lambda$.

If we allow the "number of colors" $k$ in the generating function $C_{k,j}(q)$ to be the indeterminate $k = 1-b$, and let $j$ be $\infty$, i.e, allow it to increase without bound, then $C_{k,\infty} (q)$ becomes the quantity above:

\[C_{1-b,\infty (q)} = \prod_{n=1}^\infty \sum_{i=0}^\infty \binom{k}{i} \frac{q^{i n}}{(1-q^n)^i} = \prod_{n=1}^\infty \left(1 + \frac{q^n}{1-q^n} \right)^{1-b} = \prod_{n=1}^\infty \left( \frac{1}{1-q^n} \right)^{1-b}. \]

Then choosing any whole number value for $j$ is a truncation of this series.  A natural question is whether truncations have pleasing combinatorial relationships with the hook length formula.

In the case $j=1$, the answer is yes.  Observe that truncation at $j=1$ gives

\[ C_{1-b,1} (q) = \prod_{n=1}^\infty \left(1 + \frac{(1-b)q^n}{1-q^n} \right) = \prod_{n=1}^\infty \frac{1-b q^n}{1-q^n}. \]

Since $u = k-1 = 1-b-1 = -b$, statement (3) of the theorem tells us that $$f_n = \sum_{\lambda \vdash n} (1-b)^{\ell_o(\lambda)} = \sum_{\lambda \vdash n} \prod_{{h_{ij} \in \lambda} \atop { h_{ij} = 1}} (1-\frac{b}{h_{ij}^2}).$$

Thus, $C_{1-b,1}(q)$ is exactly the hook length formula if we were restricted to only considering hooks of size 1.

For $C_{1-b,j}(q)$ with $j>1$, equality of truncations does not precisely hold, but an interesting relationship does still occur.  For more on this, see section \ref{Truncations}.

A useful combinatorial map of order $k$ on $k$-colored partitions is rotation of the colors.  When $j=1$, we have no worries about degenerate cases in which rotation through less than $k$ colors will repeat a configuration.  Letting $\nu_i(n)$ be the number of partitions of $n$ in which exactly $i$ sizes of part appear, we may state that for $n \geq 1$,

\[ c_{k,1} (n) = \sum_{i=1}^\infty k^i \nu_i (n). \]

This is, of course, equivalent to clause (3) of Theorem \ref{KOneFn}.

This immediately gives us a number of congruences.  For $n>0$, we have 

\begin{align*}
c_{k,1}(n) &\equiv 0 \, \text{ (mod } k), \\ c_{k,1}(n) &\equiv k \nu_1(n) \, \text{ (mod } k^2), \\ c_{k,1}(n) &\equiv k\nu_1(n) + k^2 \nu_2(n) \, \text{(mod } k^3), \text{ etc.}
\end{align*}

Note that, of course, $\nu_1(n) = d(n)$, the number of divisors of $n$.

It is known that overpartitions satisfy stronger congruences mod $2^i$ in certain arithmetic progressions, and congruences modulo powers of other primes as well.  Those are not explained by this elementary relation, and thus in some sense could be considered "unexpected": such a congruence contains information about uncolored partitions.

In \cite{Kim}, Byungchan Kim showed, for instance, that $\overline{p}(n)$, the number of overpartitions of $n$, satisfies

$$\overline{p}(n) \equiv 0 \, \text{ (mod } 8) \, \text{ if } \, n \neq k^2, 2k^2$$

\noindent for any integer $k$.  Since we know $$\overline{p}(n) \equiv 2 \nu_1(n) + 4 \nu_2(n) \, \text{ (mod } \, 8),$$

\noindent this is equivalent to the assertion that for $n \neq k^2$ or $2k^2$, we have $\nu_2$ and $\frac{\nu_1}{2}$ simultaneously either both even or both odd.

A rather stranger consequence follows from $$\overline{p}(9^{\alpha}(27n+18)) \equiv 0 \pmod{3}$$ \noindent from \cite{HirschSell}, which thus gives us that for $m=9^{\alpha}(27n+18)$, $$-\nu_1(m) + \nu_2(m) - \nu_3(m) + \nu_4(m) - \dots \equiv 0 \pmod{3}.$$

We will not dwell on overpartitions directly.  However, many otherwise surprising congruences for $(k,1)$-colored partitions are a result of the fact that $(k,1)$-colored partitions are congruent to overpartitions mod $k-2$:

\begin{lemma}\label{ModOverPs} $$\prod_{n=1}^\infty \left(1 + \frac{\binom{k}{1} q^n}{1-q^n} \right) = \prod_{n=1}^\infty \frac{1+(k-1)q^n}{1-q^n} \equiv_{k-2} \prod_{n=1}^\infty \frac{1+q^n}{1-q^n} = \sum_{n=0}^\infty \overline{p}(n)q^n.$$
\end{lemma}

Combined with the fact that $c_{k,1}(n) \equiv 0$ (mod $k$) for $n>1$, we obtain numerous composite congruences.  The following congruences all follow immediately from the previous lemma and facts known for overpartitions.

\begin{proposition} 
\begin{align*}
c_{5,1} (24n+19) &\equiv 0 \, \text{(mod } 15) \\
c_{11,1} (24n+19) &\equiv 0 \, \text{(mod } 99) \\
c_{29,1} (24n+19) &\equiv 0 \, \text{(mod } 783) \\
c_{5,1} (9^\alpha (27n+18)) &\equiv 0 \, \text{(mod } 15) \\
c_{7,1} (4^\alpha (40n+35)) &\equiv 0 \, \text{(mod } 35) \\
c_{17,1} (24n+19) &\equiv 0 \, \text{(mod } 51) \\
c_{17,1} (40n+35) &\equiv 0 \, \text{(mod } 85) \\
c_{17,1} (120n+115) &\equiv 0 \, \text{(mod } 255)
\end{align*}
\end{proposition}

\begin{proof} Xia and Yao prove $\overline{p}(24n+19) \equiv 0$ (mod $3^3$) in \cite{XiaYao}, giving the first three congruences.  Hirschhorn and Sellers show $\overline{p}(9^\alpha(27n+18)) \equiv 0$ (mod 3) in \cite{HirschSell}.  That $\overline{p}(40n+35) \equiv 0$ (mod 5) was conjectured by Hirschhorn and Sellers and proved by Chen and Xia; the more general $4^k(40n+35)$ was proved by W. Y. C. Chen, Sun, Wang and Zhang very recently \cite{CSWZ}, in a paper which also includes many other congruences mod 5 for the overpartition function that can be combined with modulus 7 for $c_{7,1}$.

The $c_{17,1}$ congruences follow from the fact that congruences modulo 15 reduce to congruences modulo both 3 and 5, and the combination of these gives the latter.  

Bringmann and Lovejoy showed \cite{BrLoIRMN} that infinitely many non-nested arithmetic progressions $An+B$ exist for which $\overline{p}(An+B) \equiv 0$ (mod $\ell$) for any desired prime $\ell \geq 5$, which can be employed in this fashion.
\end{proof}

Similarly to the preceding arguments, $(k,1)$-colored partitions share congruences with the usual partition function modulo $k-1$.  The process is exactly the same as the above, and so we will forbear to detail examples.

It is perhaps worthy of special mention, though, that this means $c_{2k,1}(n)$ shares the parity of overpartitions, $c_{4,1}(n)$ specifically the residue class mod 3 of the usual partition function, while $c_{3,1}(n)$ (indeed, any$c_{2k-1,1}(n)$) shares the parity of the usual partition function.

All of the above congruences have odd modulus.  The case for even modulus is slightly different, as it not only shares the parity of overpartitions, but can achieve congruences modulo greater powers of 2.  For instance, $c_{4,1}(n) \equiv 4\nu_1(n) + 16 \nu_2(n) \pmod{64}$.  Hence we digress briefly to consider $\nu_1$ and $\nu_2$ mod powers of 2.

In the former case, $\nu_1 (n) = d(n)$, the divisor function, which is perfectly understood.  If the factorization of $n$ into primes is $n = p_1^{\alpha_1} p_2^{\alpha_2} \dots$, then $$d(n) = (\alpha_1 + 1)(\alpha_2 + 1) \dots.$$

We thus observe $\nu_2$.  The quantities $\nu_i (n)$ were studied by Major MacMahon and more recently George Andrews \cite{GEA1}, who gave the generating function

$$N_k(q): = \sum_{n=0}^\infty \nu_k(n) q^n = \frac{1}{(q;q)_\infty} \sum_{m=0}^\infty \frac{(-1)^{m-k} \binom{m}{k}q^{m(m+1)/2}}{(q;q)_m}.$$

$N_2$ itself was more recently studied by Tani and Bouroubi \cite{BT}, who gave some restricted formulas.  We are interested in the 2-adic valuation of the coefficients of $N_2$.  First, we have an infinite class of progressions in which $\nu_2$ is even:

\begin{theorem}\label{V2Mod2}
\begin{align*}
\nu_2(4n+2) &\equiv 0 \pmod{2} \\
\nu_2(9n+6) &\equiv 0 \pmod{2} \\
\nu_2(25n+10) &\equiv 0 \pmod{2} \\
\dots & \\
\end{align*}
\end{theorem}

The latter two sequences are of the form $p(pn+r)$ with $r$ a quadratic nonresidue mod $p$.  For these we prove the more general theorem,

\begin{theorem} If two or more primes appear to odd order in the prime factorization of $m$, i.e. $d(m) \equiv 0$ (mod 4), then $\nu_2(m) \equiv 0$ (mod 2).
\end{theorem}

\begin{proof}
The number of distinct part sizes of a partition is invariant under conjugation.  The parity of a set of partitions closed under conjugation is equal to the parity of its self-conjugate subset.  Any self-conjugate partition into two sizes of parts can be described as follows:

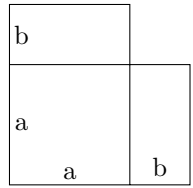
\begin{center}
\begin{tikzpicture}[scale=0.8]
\draw (2,0) -- (2,3) -- (0,3) -- (0,0) -- (3,0) -- (3,2) -- (0,2);
\draw (0.2, 1) node {a};
\draw (1,0.2) node {a};
\draw (0.2,2.5) node {b};
\draw (2.5,0.3) node {b};
\end{tikzpicture}
\end{center}

For $m$ being the integer partitioned, we thus have $$m=a^2 + 2ab = a(a+2b) = (a+b)^2 - b^2.$$ 

It is elementary that $m=4n+2$ cannot be written as a difference of squares.  Thus there are no self-conjugate partitions of $m=4n+2$ into exactly 2 part sizes, so $\nu_2(4n+2) \equiv 0$ (mod 2).

One notes that $a \equiv a+2b$ (mod 2), so any $a$ which describes a self-conjugate partition of $m$ into two sizes of part must be of the same parity as $m$: so $a$ is odd if $m$ is odd, no suitable $a$ exists if $m=4n+2$, and $a$ must be even if $4 \vert m$.  Since $b$ is determined by  $a$ and is an integer as long as $m/a - a$ is even, each $a$ will give a unique self-conjugate partition of $m$ into exactly two part sizes as long as $$a \vert m, \quad a < \sqrt{m}, \quad a \equiv m \, \text{(mod } 2), \quad s_2(a) < s_2(m) \, \text{ or both are 0}.$$

Here $s_p(n)$ is the power of p that divides $n$.  Now suppose $$m = p_a^{\alpha_0} p_b^{\alpha_1} p_1^{\beta_1} \dots p_r^{\beta_r}$$

\noindent is the prime factorization of $m$ with $\alpha_0$, $\alpha_1$ odd (as some $\beta_i$ may be).  If $p_a$, $p_b$, or any $p_i$ is 2, then the relevant exponent must be at least 2.

If $m$ is odd, all divisors of $m$ are odd and so all divisors $a \vert m$, $a < \sqrt{m}$ give self-conjugate partitions of two part sizes.  Since $m$ has two odd exponents among its prime factorization, $\sqrt{m}$ is not a factor of $m$, and since $d(m) \equiv 0$ (mod 4), half this gives $\nu_2 (m) \equiv 0$ (mod 2).

If $m$ is even, then let $$m^\prime = \frac{m}{4} = 2^{-2} p_a^{\alpha_0} p_b^{\alpha_1} p_1^{\beta_1} \dots p_r^{\beta_r}.$$

If $p_a = 2$ or $p_b = 2$, the resulting exponent $\alpha_0 - 2$ or $\alpha_1 - 2$ is still positive and odd; if some $p_i = 2$, the exponent may be 0 and is irrelevant regardless.  We still have that $d(m^\prime) \equiv 0$ (mod 4).

For any self-conjugate partition of $m$, say $a=2c$, so $a+2b = 2(c+b)$.  Then $m^\prime = c(c+b)$, $b>0$.  The parity of $c$ does not matter: any $c < \sqrt{m^\prime}$ yields a suitable factorization since $a = 2c$ has 2-valuation at least 1 less than $m$, and $c < \sqrt{m^\prime} \Longrightarrow a = 2c < \sqrt{m}$.  We thus have a resulting self-conjugate partition of $m$, and all such partitions so arise.  Thus again $\nu_2(m) \equiv 0$ (mod 2).

\end{proof}

Numerical experimentation to date has so far yielded congruences only modulo 2 and 4 for $\nu_2$ and $\nu_3$ (and none for higher $\nu_i$). We venture no conjecture as to whether congruences for other moduli might exist in large arithmetic progressions.

This argument can establish parity but in order to establish congruences mod 4 we would need an involution other than conjugation, or a different technique entirely.  This was the case for

\begin{theorem} $v_2(16n+14) \equiv 0 \pmod{4}$.
\end{theorem}

\begin{proof} For the crucial lemma to complete this proof the author cordially thanks Jeremy Rouse of Wake Forest University, who provided the modular-forms argument in response to a question on MathOverflow \cite{MOJRouse}.

MacMahon, Andrews, and Dilcher all derived the identity

$$\nu_2(n) = \frac{1}{2}\left(\sum_{k=1}^{n-1} d(k)d(n-k) - \sigma_1(n) + d(n) \right)$$

\noindent where $\sigma_1(n)$ is the sum of the divisors of $n$.  They manipulated the generating function, but for self-containment we will prove it combinatorially.

Consider a "two-sided Ferrers diagram" built by separating $n$ into $k$ and $n-k$, both nonzero, and creating two rectangles, one of area $k$ and the other of area $n-k$.  Clearly the number of such diagrams is $\sum_{k=1}^{n-1} d(k)d(n-k)$.  When the two rectangles are of the same height $j$, which is a divisor of $n$, the resulting single rectangle can arise when $k$ is anything from 1 to $\frac{n}{j}-1$.

To eliminate copies, we subtract $\frac{n}{j}$ such enumerated pairs, and as we range over all $j$ dividing $n$ this totals $\sigma_1(n)$ subtractions.  But this is 1 too many, so we re-add one for each such $j$, i.e. $d(n)$.  Finally, from the resulting enumeration of pairs of different heights, we take the half of such pairs where the smaller is on the right, and obtain correct Ferrers diagrams of partitions into two sizes of parts.

We now show $$m=16n+14 \Longrightarrow \sigma_1(m) \equiv 0 \pmod{8}.$$  We have $m=2(8n+7)$, and $8n+7$ has exactly one of "the exponents of primes 7 mod 8 in its factorization sum to an odd number," or "the exponents of primes 3 and 5 mod 8 in its factorization both sum to odd numbers."  If the former is the case, let $p \equiv 7 \pmod{8}$ be such a prime.  Then consider factors $x$, where $x$ is a factor of $m$ not divisible by 2 and in which $p$ appears to even order (possibly 0).  Group these with $x$, $2x$, $px$, and $2px$, summing to $24x$ modulo 8.  In the latter case, designate suitable primes 3 and 5 mod 8, and the grouping is $x$, $2x$, $3x$, $5x$, $6x$, $10x$, $15x$, and $30x$, summing to $72x$.

Thus for $m \equiv 14 \pmod{16}$, $$\nu_2(m) \equiv \frac{1}{2}\left(\sum_{k=1}^{m-1} d(k)d(m-k) + d(m) \right) \pmod{4},$$

\noindent and we observe that for $m \equiv 6 \pmod{8}$, both $d(m)$ and $\sum_{k=1}^{m-1} d(k)d(m-k)$ are divisible by 4: the former since two primes (2 and one which is 3 mod 4) appear to odd order in the factorization of $m$, and the latter since $m$ is neither a square nor the sum of two squares, so at least one of $d(k)$ and $d(m-k)$ is even, and each product appears twice in the sum, except the middle term $d\left( \frac{m}{2} \right)^2$, which is a multiple of 4.

Thus to prove the theorem it suffices to show the lemma:

\begin{lemma}\label{RouseLemma} (Rouse) For $m \equiv 6 \pmod{8}$, $d(m) \equiv \sum_{k=1}^{m-1} d(k)d(m-k) \pmod{8}.$
\end{lemma}

It follows from the description above that $$\sum_{k=1}^{m-1} d(k) d(m-k) = 2 \sum_{k=1}^{\frac{m-2}{2}} d(k) d(m-k) + d \left( \frac{m}{2} \right)^2.$$

Whem $m \equiv 6 \pmod{8}$, $m = 2(4j+3)$, and the prime factorization of $(4j+3)$ contains an odd number of primes which are 3 mod 4.  Depending on the parity of all primes in its factorization, $d(4j+3)$ is either 2 or 4 mod 4, and thus $$d(m) \equiv d\left( \frac{m}{2} \right)^2 \pmod{8}.$$

Thus we may subtract these from both sides of the claimed identity, and it suffices to prove

$$\sum_{k=1}^{\frac{m-2}{2}} d(k) d(m-k) \equiv 0 \pmod{4}.$$

There are no odd terms, since $m$ is not the sum of two squares, and therefore we wish to show that there are an even number of terms that are not multiples of 4.

The only terms that are not multiples of 4 are those in which $k$ or $m-k$ is square, and the other term is 2 mod 4, i.e. $m-k$ or $k$ respectively is $p y^2$ for $p$ a prime, with $s_p(y) \equiv 0 \pmod{2}$.  Since $m \equiv 6 \pmod{8}$, when $m=x^2 + p y^2$ either $x$ is even, $p=2$ and $y$ is odd, or $x$ is odd, $y$ is odd and $p \equiv 5 \pmod{8}$ (since $x^2$ and $y^2$ are both 1 mod 8).

We note that the latter conditions are both necessary and sufficient for such an $n = py^2 \equiv 5 \pmod{8}$ to have $\sigma_1(n) \equiv 2 \pmod{4}$.  Otherwise, $\sigma_1(n) \equiv 0 \pmod{4}$.  That is,

$$\frac{1}{2} \sum_{n=0}^\infty \sigma_1(8n+5) q^{8n+5} \equiv \sum_{{p \equiv 5 \pmod{8}} \atop {y \geq 1, 2 \vert s_p(y)}} q^{p y^2} \pmod{2}.$$

We have for all odd numbers,

$$F(q) := \sum_{n=0}^\infty \sigma_1(2n+1) q^{2n+1} \equiv \sum_{n =1}^\infty q^{(2n+1)^2} \pmod{2}.$$

The remainder of the proof employs the properties of modular forms.  We will not include all the definitions here, but will summarize the useful facts:

\begin{itemize}
\item A modular form is said to be of weight $k$ for $\Gamma_0(N)$, a certain subgroup of $\mathfrak{sl}_2$. Its \emph{level} is the minimum possible $N$.  Such a form is also of weight $k$ for any $\Gamma_0(cN)$, $c \in \mathbb{N}$.
\item Modular forms of a given weight for $\Gamma_0(N)$ form a vector space over $\mathbb{C}$, and the substitutions $q \rightarrow q^c$ for $c \in \mathbb{N}$ send forms of weight $k$ for $\Gamma_0(N)$ to forms of weight $k$ for $\Gamma_0(cN)$.
\item The product of two modular forms for $\Gamma_0(N)$ of weights $k$ and $\ell$ is a modular form for $\Gamma_0(N)$ of weight $k+\ell$.
\item Given the weight and level of a form $f(q)$, there is a $c$ depending on these known as the \emph{Sturm bound} for which, if all coefficients of $q^i$ in $f(q)$ with $i \leq c$ are divisible by a prime $p$, then all coefficients of $f$ are so divisible.
\end{itemize}

The two modular forms necessary for our proof are

$$F(q) = \sum_{n=0}^\infty \sigma_1(2n+1) q^{2n+1} \quad \text{ and } \quad G(q) := \sum_{n=0}^\infty \sigma_1(8n+5) q^{8n+5}$$

\noindent where $F(q)$ has weight 2 on $\Gamma_0(4)$, and $G(q)$ is a modular form of weight 2 on $\Gamma_0(64)$.  Thus $H(q) = F(q)G(q) + F(q^4)F(q^2)$ is a modular form of weight 4 on $\Gamma_0(64)$.  Letting $r_{1,p}(n)$ denote the number of representations of $n$ of the form $x^2+py^2$, $s_p(y) \equiv 0 \pmod{2}$, we have

\begin{multline}
H(q) \equiv \sum_{{x,y \geq 1 \text{ odd},} \atop {p \equiv 5 \pmod{8} \text{ prime}\, 2 \vert s_p(y)}} q^{x^2 + py^2} + \sum_{{x,y \geq 1, \, x \equiv 2 \pmod{4}} \atop {y \equiv 1 \pmod{2}}}q^{x^2 + 2y^2} \pmod{2} \\
= \sum_{n=0}^\infty r_{1,p} (8n+6) q^n.
\end{multline}

The Sturm bound is 32, and calculation gives us that the coefficients of $H(q)$ up to $q^{32}$ are indeed even: hence all are.

Since the number of representations of $8n+6$ of the form $x^2 + py^2$, $s_p(y) \equiv 0 \pmod{2}$ is even, there is an even number of terms 2 mod 4 in $\sum_{n=1}^{\frac{m-2}{2}} d(k) d(m-k)$, and so the sum is 0 mod 4.  This concludes the proof.
\end{proof}

The corollaries of these theorems for $c_{k,1}$ for $k$ even are several.  Among others, we have:

\begin{corollary} 
\begin{align*}c_{2k,1} (8n+6) &\equiv 0 \pmod{8} \\
c_{2k,1} (16n+14) &\equiv 2k d(16n+14) + 8 k^3 \nu_3 (16n+14) \pmod{16} \\
c_{4k,1} (16n+14) &\equiv 0 \pmod{16} \\
\nu_3(16n+14) &\equiv \frac{1}{4}d(16n+14) \pmod{2} \\
c_{2k,1} (p(pn+r)) &\equiv 0 \pmod{8} \, (r \, \text{not a quadratic residue mod } \, p) \\
\end{align*}
\end{corollary}

(We note that the congruences for $c_{2,1}$ have independently appeared in preprint during the writing of this paper, proved by dissection of generating functions: \cite{CHSZ}.  They also establish $c_{2,1}(16n+14) \equiv 0 \pmod{16}$, which with this technique would require information on $\nu_3$, and other congruences for overpartitions modulo powers of 2.)

\begin{proof}
From the previous argument, $\nu_2(8n+6) \equiv \frac{1}{2} \sigma_1(8n+6) \pmod{4}$.  Hence, at worst if $k$ is odd, $$c_{2k,1} (8n+6) \equiv 2 d(8n+6) + 2 \sigma_1(8n+6) \pmod{8}.$$

But $d(8n+6) \equiv 0 \pmod{4}$ and $\sigma_1(8n+6) \equiv 0 \pmod{4}$ by grouping factors $x$, $2x$, $3x$, and $6x$ similar to the earlier argument.  Hence $c_{2k,1} (8n+6) \equiv 0 \pmod{8}$.

For the second clause, we have $$c_{2k,1} (16n+14) \equiv 2k d(16n+14) + 4k^2 \nu_2(16n+14) + 8k^3 \nu_3(16n+14) \pmod{16}.$$

Removing $\nu_2(16n+14) \equiv 0 \pmod{4}$ by the theorem, have the remaining terms.  The third clause immediately follows.

The fourth clause follows from $$c_{2,1} (16n+14) \equiv 0 \pmod{16},$$

\noindent which was established in \cite{CHSZ}.

The last clause follows from Theorem \ref{V2Mod2}.
\end{proof}

\noindent \textbf{Remarks:} There are several interesting points about the proof of Lemma \ref{RouseLemma}.  We have used the technology of modular forms to prove a congruence for a function which is not, itself, a modular form.  The entry of representations by sums of squares echoes the appearance of this topic in overpartition theory (see for instance \cite{Kim}), so there may be further unexplored connections between these representations and $\nu_2$, and possibly between $\nu_i$ and other forms.

Since $16n+14 \equiv 2 \pmod{4}$, it has no self-conjugate partitions into two sizes of part, and the congruence mod 4 could be proven more directly by finding an involution on conjugate pairs with no fixed points.  It would probably be interesting to find this.

It also appears to be the case that 

\begin{conjecture} $\nu_2(36n+30) \equiv 0 \pmod{4}$ and $\nu_3(36n+30) \equiv 0 \pmod{2}$.
\end{conjecture}

\noindent \textbf{Remark:} The proof for $\nu_2(36n+30)$ involves more cases for the residues of $x$ and $y$ mod 36 in $x^2+py^2$ than for $\nu_2(16n+14)$, but does not otherwise seem to be any deeper in technique, and we hope to produce it shortly.  Proofs for $\nu_3$ would require information on $\sigma_2(n)$ and therefore a different approach.

\section{$(k,k-1)$-colored partitions}\label{KMinus1}

Another special case of considerable interest is the case of $(k,k-1)$-colored partitions.  This may be considered a different generalization of overpartitions, which are the $k=2$ case. Their generating function is

\begin{multline}C_{k,k-1}(q) = \prod_{n=1}^\infty \frac{\left( \sum_{i=0}^k (1-q^n)^{k-i} \binom{k}{i} q^{in} \right) (1-q^n)^{-1} - \frac{q^{kn}}{1-q^n}}{(1-q^n)^{k-1}} \\
 = \prod_{n=1}^\infty \frac{(q^n - (1-q^n))^k - q^{kn}}{(1-q^n)^k} = \prod_{n=1}^\infty \frac{1-q^{kn}}{(1-q^n)^k}.\end{multline}
 
This is an $\eta$-quotient, i.e. a quotient of products of functions $${f_t} := \prod_{n=1}^\infty (1-q^{tn}).$$ By a well-known theorem of Gordon, Hughes and Newman (\cite{GH}, \cite{New}) on $\eta$-quotients, it is always a modular form of negative weight $\frac{1}{2}(1-k)$ and level $24k$ (or any level this divides).   Together these two facts promise us the existence of many congruences and give us numerous tools in the literature to prove them.

Here is a small sample of available congruences:

\begin{theorem} 
\begin{align*}
c_{3,2}(3n+2) &\equiv 0 \pmod{9} \\
c_{3,2}(4n+2) &\equiv 0 \pmod{9} \\ 
c_{5,4}(4n+2) &\equiv 0 \pmod{20} \\
c_{5,4}(2n) &\equiv 0 \pmod{10} \, \text{for } \, n \neq 0
\end{align*}
\end{theorem}

\begin{proof}

For $c_{3,2}(3n+2)$, we note that whether a part size has 1 or 2 colors, we may rotate the colors employed to obtain a group of size 3 for that part.  Thus the only terms that contribute mod 9 arise from partitions into 1 size of part, i.e. the divisors of $m$.

The partition $\lambda = m^1$ (in frequency notation) gives 3 partitions, and the partition $\lambda = d^{m/d}$ gives $3 + 3(\frac{m}{d}-1)$ partitions, arising from assigning each of the 3 colors once, plus dividing $\frac{m}{d}$ into two nonzero sections and assigning the three pairs of colors.  This gives a total of

$$c_{3,2}(m) \equiv \sum_{d \vert m} (3 + 3(\frac{m}{d}-1)) = 3 \sigma_1(m) \pmod{9}.$$

When $m=3n+2$, there is some prime $p_0$ in the factorization of $m$ which has $p_0 \equiv 2 \pmod{3}$ and which appears to odd exponent $\alpha_0$.  Then

$$\sigma_1(3n+2) = \frac{{p_0}^{\alpha_0+1}}{p_0-1} \cdot \sigma_1 \left( \frac{m}{{p_0}^{\alpha_0}} \right) \equiv_3 \frac{4^{(\alpha_0+1)/2}-1}{2-1} \cdot \sigma_1 \left( \frac{m}{{p_0}^{\alpha_0}} \right) \equiv 0 \pmod{3}.$$

So $c_{3,2} (3n+2) \equiv 0 \pmod{9}$.

For the remaining congruences, we employ several $\eta$-quotient dissection identities, some classical and some due to Xia and Yao \cite{XY}.

For $c_{3,2}(4n+2)$, 

$$\frac{f_3}{f_1^3} = \frac{f_4^6 f_6^3}{f_2^9 f_{12}^2} + 3q \frac{f_4^2 f_6 f_{12}^2}{f_2^7} \equiv_9 \frac{f_4^6}{f_{12}^2}+3q \frac{f_4^2 f_{12}^2}{f_2^4}.$$

The first equality can be found in \cite{XY}, and the congruence arises from the identity $(1-q)^9 \equiv_3 (1-q^3)^3$.  The latter term of the right-hand side gives only odd exponents in $q$, and the former gives only $q^{4n}$, so all terms $q^{4n+2}$ have coefficients that are 0 mod 9.

For $c_{5,4}$,

\begin{multline}
\frac{f_5}{{f_1}^5} = \frac{f_5}{f_1} \cdot \frac{1}{{f_1}^4} = \left( \frac{f_8 {f_{20}}^2}{{f_2}^2 f_{40}}+q \frac{{f_4}^3 f_{10} f_{40}}{{f_2}^3 f_8 f_{20}} \right) \cdot \left( \frac{{f_4}^{14}}{{f_2}^{14} {f_8}^4} + 4q \frac{{f_4}^2 {f_8}^4}{{f_2}^{10}} \right) \\
\equiv_4 \frac{{f_{20}}^2 {f_4}^{14}}{f_{40}{f_8}^3} \cdot \frac{1}{{f_2}^{16}} + (\text{terms of odd exponent}) \equiv_4  \frac{{f_{20}}^2 {f_4}^{6}}{f_{40}{f_8}^3} + \dots
\end{multline}

\noindent where in the last line we employed the identity ${f_a}^4 \equiv_4 {f_{2a}}^2$.

Hence all terms $q^{4n+2}$ have coefficient divisible by 4.  Since $1-q^{5n} \equiv_5 (1-q^n)^5$, all nonconstant terms are divisible by 5 (combinatorially, we can always rotate the colors of the smallest part size), so the congruence holds.  

Indeed, since $(1-q^2) \equiv_2 (1-q)^2$, every $q^{2t}$ term except for $t=0$ is even; everything else cancels mod 2, and thus the third congruence follows.

\end{proof}

Many other dissections may be employed to prove similar congruences.

\section{Results for general $j$}\label{GeneralJ}

For $j\not\in \{1, k-1\}$, there are other congruences which seem to arise mainly when the generating function shares a congruence with $k$-colored partitions or other classical generating functions.  We list a few of these here.

\begin{theorem}
\begin{align*}
c_{7,2} (5n+2,3,4) &\equiv 0 \pmod{35} \\
c_{7,2} (8n+4,6) &\equiv 0 \pmod{14}
\end{align*}
\end{theorem}

\begin{proof}
We first note that $c_{7,2} \equiv 0 \pmod{7}$ by the usual argument.  Expanding the generating function, we find

$$C_{7,2} (q) = \prod_{n=1}^\infty \left( 1 + \frac{7 q^n}{1-q^n} + \frac{21q^{2n}}{(1-q^n)^2} \right) = \prod_{n=1}^\infty \frac{1+5q^n+15q^{2n}}{(1-q^n)^2} .$$

Thus $c_{7,2}(m) \equiv c_{2,2} (m) \pmod{5}$, i.e. the number of unrestricted 2-colored partitions of $m$.  It is known that $c_{2,2}(5n+j) \equiv 0 \pmod{5}$ for $j=2,3,4$ (for instance, square Hirschhorn's decomposition of the partition function mod 5), and first line follows.

\noindent \textbf{Remark:} Observe that when $k-j$ is a prime with $k-j > j$, it is always the case that $\binom{k-j+i-1}{i}$ will be divisible by $k-j$ for $i \neq 0$, since $i \leq j$.  Then $c_{k,j}$ will share a congruence modulo $k-j$ with the unrestricted $j$-colored partitions when one exists.

For the second line, we observe that 

$$\prod_{n=1}^\infty \frac{1+5q^n+15q^{2n}}{(1-q^n)^2} \equiv_2 \prod_{n=1}^\infty \frac{1+q^n+q^{2n}}{(1-q^n)^2} = \prod_{n=1}^\infty \frac{1-q^{3n}}{(1-q^n)^3}.$$

We now 2-dissect this generating function using identities due to various authors (collected in \cite{XiaYao}), especially:

$$\frac{f_3}{f_1^3} = \frac{f_4^6 f_6^3}{f_2^9 f_{12}^2} + 3q \frac{f_4^2 f_6 f_{12}^2}{f_2^7} \quad \text{ and } \quad \frac{f_1^3}{f_3} = \frac{f_4^3}{f_{12}} -3q \frac{f_2^2 f_{12}^3}{f_4 f_6^2}.$$

We dissect $C_{7,2}(q)$ thus:

\begin{multline*}
\frac{f_3}{f_1^3} = \frac{f_4^6 f_6^3}{f_2^9 f_{12}^2} + (\text{odd terms}) \equiv_2 \left( \frac{f_4^3}{f_{12}} \right)^2 \left( \frac{f_6}{f_2^3} \right)^3 + \cdots \\
\equiv_2 \left( \frac{f_{12}}{f_4^3} \right) \left( \frac{f_{16}^6}{f_{48}^2} + q^8 \frac{f_8^4 f_{48}^6}{f_{16}^2 f_{24}^4} \right) \left( \frac{f_8^6}{f_{24}^2} + q^2 f_8^3 f_{24} + q^4 f_{24}^4 + q^6 \frac{f_{24}^7}{f_8^3}  \right) \\
\equiv_2 \left( \frac{f_{16}^6 f_{24}^3}{f_8^9 f_{48}^2} + q^4 \frac{f_{16}^2 f_{24} f_{48}^2}{f_8^7} \right) \left( \frac{f_{16}^6}{f_{48}^2} + q^8 \frac{f_8^4 f_{48}^6}{f_{16}^2 f_{24}^4} \right) \left( \frac{f_8^6}{f_{24}^2} + q^2 f_8^3 f_{24} + q^4 f_{24}^4 + q^6 \frac{f_{24}^7}{f_8^3}  \right) \\
=: (B + q^4 C)(A)(D + q^2 E + q^4 F + q^6 G)
\end{multline*}

The last line names the individual terms for convenience; both terms in $A$ are strictly functions of $q^8$, so we combine them.

The terms of exponent $q^{8n+4}$ in the sum are $A(BF + CD)$.  By making repeated use of the congruence $f_1^2 \equiv_2 f_2$, we have $BF \equiv_2 CD$:

$$\frac{f_{16}^6 f_{24}^7}{f_8^9 f_{48}^2} \equiv_2 \frac{f_{16}^2 f_{24}^8}{f_8 f_{24} f_{48}^2} \equiv_2 \frac{f_{16}^2 f_{48}^2}{f_{24} f_8}.$$

Hence the sum $BF+CD \equiv_0 \pmod{2}$, and the congruence is established.

Similarly, the terms of exponent $q^{8n+6}$ arise from $A(BG+CE)$, and $BG \equiv_2 CE$, so both claims hold.
\end{proof}

\section{Truncations and the Hook Formula}\label{Truncations}

We conclude with further observations on the hooklength formula.

We showed in Section \ref{K1} that in $C_{1-b,1}(q)$, the coefficient on $q^n$ was exactly a restriction of the Han/Nekrasov-Okounkov hook length formula to considering only hooks of size 1.  Unfortunately, truncations at higher $j$ do not precisely match the hook length formula restricted to hooks of size $j$, but a combinatorially interesting relationship of some kind does seem plausible.  If we truncate at $j=2$, we can obtain a more limited relationship.  The generating function is:

\begin{multline*} C_{1-b,2} (q) = \prod_{n=1}^\infty \left( 1 + \frac{(1-b) q^n}{1-q^n} + \frac{\binom{1-b}{2} q^{2n}}{(1-q^n)^2} \right) \\ = \prod_{n=1}^\infty \frac{1}{(1-q^n)^2} \left( 1-(b+1) q^n + \left(\frac{b^2+b}{2}\right) q^{2n}\right).
\end{multline*}

Expanding $\prod_{n=1}^\infty \frac{1}{(1-q^n)^2} = \sum_{n=0}^\infty (n+1)q^n$ and gathering terms by partitions, we obtain

\[ C_{1-b,2}(q) = \sum_{n=0}^\infty q^n \left( \sum_{\lambda = 1^{\lambda_1} 2^{\lambda_2} \dots \vdash n} (1-b)^{\ell_0(\lambda)} \prod_{\lambda_j \geq 2} (1-\frac{b}{2}(\lambda_j-1)) \right). \]

We can now claim

\begin{theorem} Let $\ell_0(\lambda)$, $\lambda_j$ be defined as in $C_{1-b,2}$ above.  Then $$\sum_{\lambda \vdash n} \prod_{{h_{ij} \in \lambda} \atop {h_{ij} = 1,2}} \left(1-\frac{b}{{h_{ij}}^2}\right) \quad \text{ and } \quad \sum_{\lambda \vdash n} (1-b)^{\ell_0(\lambda)} \prod_{\lambda_j \geq 2} \left(1-\frac{b}{2}\right)$$

\noindent have the same constant and linear term in $b$.
\end{theorem}

That is, the hook length formula considering only hooks of size 1 or 2 can be matched by $C_{1-b,2}$ up to the linear term in $b$ if we reduce any $\lambda_j > 2$ to just 2.

\begin{proof}

\begin{figure}
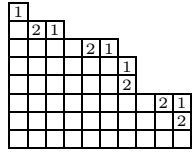

$\tiny\young(1:::::::::,\hfil 21:::::::,\hfil\hfil\hfil\hfil21::::,\hfil\hfil\hfil\hfil\hfil\hfil1:::,\hfil\hfil\hfil\hfil\hfil\hfil2:::,\hfil\hfil\hfil\hfil\hfil\hfil\hfil\hfil21,\hfil\hfil\hfil\hfil\hfil\hfil\hfil\hfil\hfil2,\hfil\hfil\hfil\hfil\hfil\hfil\hfil\hfil\hfil\hfil)
$
\caption{A partition with hooks of size 1 and 2 marked.}
\end{figure}

Consider a partition in which we observe the hooks of size 1 or 2.  Hooks of either of these sizes can only appear in the outermost squares of a Ferrers diagram.  Say there are $\ell_0(\lambda)$ hooks of size 1, $a_1$ hooks of size 2 which are on the top row of a part size, and $a_2$ hooks of size 2 which are in a part below a corner.  In the figure above, $a_1=3$ and $a_2=2$.  Note that the partition in the figure is not self-conjugate.  In a self-conjugate partition, $a_1 = a_2$.

Suppose $\lambda$ is not self-conjugate.  Then the contribution of $\lambda$ plus that of $\lambda^\prime$ to the terms of the two expressions of the theorem is, for the hook length formula restricted to hooks of size 1 and 2,

\begin{multline*}2 \cdot (1-b)^{\ell_0(\lambda)} \left(1-\frac{b}{4}\right)^{a_1 + a_2} = (1-b)^{\ell_0(\lambda)} \left( 2 - 2 \cdot \binom{a_1 + a_2}{1} \frac{b}{4} +O(b^2) \right) \\
= (1-b)^{\ell_0(\lambda)} \left( 2 - (a_1 + a_2)\frac{b}{2} +O(b^2) \right)
\end{multline*}

and, for $C_{1-b,2}$ with $\lambda_j$ restricted to 2,

$$(1-b)^{\ell_0(\lambda)} \left( \left(1-\frac{b}{2}\right)^{a_1} + \left(1-\frac{b}{2}\right)^{a_2} \right) = (1-b)^{\ell_0(\lambda)} \left( 2- (a_1 + a_2) \frac{b}{2} + O(b^2) \right).$$

For a self-conjugate partition, we have for the hook length formula,

\begin{multline*}(1-b)^{\ell_0(\lambda)} \left(1-\frac{b}{4}\right)^{a_1 + a_2} = (1-b)^{\ell_0(\lambda)} \left( 1 - \binom{2a_1}{1} \frac{b}{4} +O(b^2) \right) \\
= (1-b)^{\ell_0(\lambda)} \left( 1 - a_1 \frac{b}{2} +O(b^2) \right)
\end{multline*}

and, for $C_{1-b,2}$,

$$(1-b)^{\ell_0(\lambda)} \left(1-\frac{b}{2}\right)^{a_1} = (1-b)^{\ell_0(\lambda)} \left( 1- a_1 \frac{b}{2} + O(b^2) \right).$$

Thus the two expressions are equal up to the linear term.

\end{proof}

Numerical calculation suggests that this also holds for the 3-truncations and 4-truncations: i.e. the constant and linear coefficients appear to match in

$$\sum_{\lambda \vdash n} \prod_{{h_{ij} \in \lambda} \atop {h_{ij} = 1,2,3}} \left(1-\frac{b}{h_{ij}^2}\right) \quad \text{ and } \quad \sum_{\lambda \vdash n} (1-b)^{\ell_0(\lambda)} \prod_{\lambda_j \geq 1} \left( 1 - \frac{b}{2}(\lambda_j-1) +\frac{b^2+b}{6}\binom{\lambda_j-1}{2} \right) $$

\noindent in which if $\lambda_j > 3$, we declare $\lambda_j = 3$, and the analogous expression for $j=4$.  (In expanding the third term we used the identity $\prod_{n=1}^\infty \frac{1}{(1-q^n)^3} = \sum_{n=0}^\infty \binom{n+2}{2} q^n$.)

The general conjecture to which we are led is:

\begin{conjecture} For any $n$ and fixed $m$, the linear terms are equal in the truncated hook length expression $$\sum_{\lambda \vdash n} \prod_{{h_{ij} \in \lambda} \atop {h_{ij} \leq m}} \left(1-\frac{b}{h_{ij}^2}\right) $$ and the restricted $C_{1-b,j}$ sum \begin{multline*} C_{1-b,m}^\prime (q) = \sum_{n=0}^\infty q^n \sum_{\lambda \vdash n} \prod_{\lambda_j \geq 1} \left( \sum_{i=0}^{\infty} \binom{1-b}{i+1} \binom{min(\lambda_j,m)-1}{i} \right) \\ = \sum_{n=0}^\infty q^n \sum_{\lambda \vdash n} \prod_{\lambda_j \geq 1} \left(1-\frac{b}{1} \right) \cdots \left( 1-\frac{b}{min(\lambda_j,m)}\right).\end{multline*}

\end{conjecture}

It is entirely possible that this is not the best relationship to be found.  Still, the coefficients of each $q^n$ in $C_{1-b,j}$ necessarily approach the hook length formula as $j$ grows.  It would certainly be interesting to formalize a relationship between suitable restrictions of the hook length formula and $(1-b,j)$-colored partitions which makes it clear that the latter approach the former and eventually stabilize there.  One fruitful line of investigation might be to determine what correction term is needed by the unmodified truncation to match the hook length formula.

In attempting to find such a correction factor for $C_{1-b,2}$, we made the following numerical observation: it appears to be the case that 

\begin{conjecture} For any $n$, the constant, $b$, and $b^2$ coefficients match in $$\sum_{\lambda \vdash n} \prod_{{h_{ij} \in \lambda} \atop {h_{ij} = 1,2}} \left( 1-\frac{b}{h_{ij}^2} \right)$$ and $$\sum_{\lambda \vdash n} \left[ \lambda_4 \frac{b^2 - b^3}{16} + (1-b)^{\ell_0(\lambda)} \prod_{\lambda_j > 1} \left(1-\frac{b}{2} \right) \right].$$
\end{conjecture}

Here $\lambda_4$ is the number of 4s in the partition $\lambda$, but $\sum_{\lambda \vdash n} \lambda_4$ is also known to be the number of frequencies in $\lambda$ that are at least 4, i.e. $\sum_{\lambda \vdash n} \sum_{\lambda_i \geq 4} 1$, which may be more relevant to a proof.

Of course, the quadratic coefficients would still match if we subtracted only $\lambda_4 \frac{b^2}{16}$, but note that $b^2-b^3$ divides the difference between the coefficients of the two polynomials, since $(1-b)$ divides both and the difference appears at the $b^2$ term.  Hence the best compensation factor is almost certainly a multiple of this.

\end{document}